\documentclass[11pt,oneside]{amsart}
\usepackage{hyperref}
\usepackage[a4paper,margin=3cm]{geometry}
\usepackage{amssymb}
\usepackage[T1]{fontenc}
\usepackage[utf8]{inputenc}
\usepackage[color=yellow]{todonotes}
\usepackage{enumerate}
\usepackage{tikz}
\usetikzlibrary{fit,automata,positioning,calc}

\usepackage{thmtools, thm-restate}
\declaretheorem[name=Theorem]{thm}

\declaretheorem[name=Conjecture,sibling=thm]{conj}
\declaretheorem[name=Problem,sibling=thm]{prob}
\declaretheorem[name=Claim,sibling=thm]{clm}

\newcommand{\rg}{\operatorname{rg}}

\title{Short rainbow cycles in graphs and matroids}

\author{Matt DeVos}
\address[M.~DeVos]{Department of Mathematics, Simon Fraser University, Canada}
\email{mdevos@sfu.ca}

\author{Matthew Drescher}
\address[M.~Drescher]{D\'epartement de Math\'ematique, Universit\'e Libre de Bruxelles, Belgium}
\email{knavely@gmail.com}

\author{Daryl Funk}
\address[D.~Funk]{Department of Mathematics, Douglas College, Canada}
\email{dfunk@sfu.ca}

\author{Sebastián González Hermosillo de la Maza}
\address[S.~González]{Department of Mathematics, Simon Fraser University, Canada}
\email{sga89@sfu.ca}

\author{Krystal Guo}
\address[K.~Guo]{D\'epartement de Math\'ematique, Universit\'e Libre de Bruxelles, Belgium}
\email{guo.krystal@gmail.com}

\author{Tony Huynh}
\address[T.~Huynh]{School of Mathematics, Monash University, Australia}
\email{tony.bourbaki@gmail.com}

\author{Bojan Mohar}
\address[B.~Mohar]{Department of Mathematics, Simon Fraser University, Canada \newline 
On leave from IMFM, Department of Mathematics, University
of Ljubljana}
\email{mohar@math.sfu.ca}

\author{Amanda Montejano}
\address[A.~Montejano]{Facultad de Ciencias, UNAM campus Juriquilla, México}
\email{amandamontejano@ciencias.unam.mx}

\thanks{M.\ Drescher, K.\ Guo, and T.\ Huynh were supported  by ERC Consolidator Grant 615640-ForEFront.  T.\ Huynh is also supported by the Australian Research Council.}
\thanks{B.~Mohar was supported in part by the NSERC Discovery Grant R611450 (Canada)
and by the ARRS Research Project J1-8130 (Slovenia).}

\begin{document}

\begin{abstract}
Let $G$ be a simple $n$-vertex graph and $c$ be a colouring of $E(G)$ with $n$ colours, where each colour class has size at least $2$.   We prove that $(G,c)$ contains a rainbow cycle of length at most $\lceil \frac{n}{2} \rceil$, which is best possible.  Our result settles a special case of a strengthening of the Caccetta-Häggkvist conjecture, due to Aharoni.  We also show that the matroid generalization of our main result also holds for cographic matroids, but fails for binary matroids.
\end{abstract}

\maketitle

\section{Introduction}
In 1978, Caccetta and Häggkvist~\cite{CH78} made the following conjecture.  

\begin{conj}[Caccetta-Häggkvist] \label{conj:CH}
For all positive integers $n,r$, every simple $n$-vertex digraph with minimum outdegree at least $r$ contains a directed cycle of length at most $\lceil \frac{n}{r} \rceil$.
\end{conj}

A digraph $D$ is \emph{simple} if for all $u,v \in V(D)$ there is at most one arc from $u$ to $v$.
The Caccetta-Häggkvist conjecture has proven to be a notoriously difficult problem. For example, the case $r=\frac{n}{3}$ has received considerable attention~\cite{CH78, GSS92, bondy97, shen98, HHK07, razborov13, lichiardopol13,  HKN17}, but still remains open.  See Sullivan~\cite{sullivan06} for a summary of partial results. 

Although there has been a lot of progress on approximate versions, Conjecture~\ref{conj:CH} is known to hold exactly for only a few values of $r$. The case $r=2$ was actually proved by Caccetta and Häggkvist~\cite{CH78}.  

\begin{thm}[\cite{CH78}] \label{directed}
Every simple $n$-vertex digraph with minimum outdegree at least $2$ contains a directed cycle of length at most $\lceil \frac{n}{2} \rceil$.
\end{thm}

The case $r=3$ was settled positively by Hamidoune~\cite{hamidoune87}, and $r \in \{4,5\}$ by Hoàng and Reed~\cite{HR87}. 

Given a graph $G$ and a colouring $c$ of $E(G)$, we say that a subgraph $H$ of $G$ is \emph{rainbow} if no two edges of $H$ are of the same colour.  
Aharoni (see~\cite{ADH19}) recently proposed the following strengthening of the Caccetta-Häggkvist conjecture. 

\begin{conj}[Aharoni] \label{conj:aharoni}
Let $G$ be a simple $n$-vertex graph and $c$ be a colouring of $E(G)$ with $n$ colours, where each colour class has size at least $r$.  Then $(G,c)$ contains a rainbow cycle of length at most $\lceil \frac{n}{r} \rceil$.
\end{conj}

In fact, we now show that the following weakening of Aharoni's conjecture implies the  Caccetta-Häggkvist conjecture.

\begin{conj} \label{conj:weakaharoni}
Let $G$ be a simple $n$-vertex graph and $c$ be a colouring of $E(G)$ with $n$ colours, where each colour class has size at least $r$.  Then $(G,c)$ contains a cycle $C$ of length at most $\lceil \frac{n}{r} \rceil$ such that no two incident edges of $C$ are the same colour. 
\end{conj}

\begin{proof}[Proof of Conjecture~\ref{conj:CH}, assuming Conjecture~\ref{conj:weakaharoni}]
Let $D$ be a simple digraph of order $n$ and minimum outdegree at least $r$. Let $G$ be the graph obtained from $D$ by forgetting the orientations of all arcs.  Let $V(G)=[n]$ and colour $ij \in E(G)$ with colour $i$ if $(i,j) \in E(D)$.  Clearly, this colouring uses $n$ colours.  Moreover, since $D$ has minimum outdegree at least $r$, each colour class has size at least $r$.  Therefore, by Conjecture~\ref{conj:weakaharoni}, $G$ contains a properly edge-coloured cycle $C$ of length at most $\lceil \frac{n}{r} \rceil$. Let $\vec{C}$ be the subdigraph of $D$ corresponding to $C$.  We claim that $\vec{C}$ is a directed cycle.  If not, then there exists $i,j,k \in V(\vec{C})$ such that $(j,i) \in E(\vec{C})$ and $(j,k) \in E(\vec{C})$.  Thus, the two edges of $C$ incident to vertex $j$ are the same colour. This contradicts that $C$ is properly edge-coloured.  
\end{proof}

Our main theorem is that Aharoni's conjecture holds for $r=2$.

\begin{thm} \label{thm:main}
Let $G$ be a simple $n$-vertex graph and $c$ be a colouring of $E(G)$ with $n$ colours, where each colour class has size at least $2$.  Then $(G,c)$ contains a rainbow cycle of length at most $\lceil \frac{n}{2} \rceil$.
\end{thm}

The rest of the paper is organized as follows.  In Section~\ref{sec:proof}, we prove our main theorem.  We show that our bound is tight in Section~\ref{sec:tight}, and that there is a sharp increase in the `rainbow girth' as the number of colours decreases from $n$.  In Section~\ref{sec:matroids}, we show that the natural matroid generalization of Theorem~\ref{thm:main} holds for cographic matroids, but fails for binary matroids. 
We conclude with some open problems in Section~\ref{sec:open}.

\section{Proof of the Main Theorem} \label{sec:proof}

In this section, we prove Theorem~\ref{thm:main}.  Before proceeding, we require some basic definitions. Let $G$ be a graph.
A \emph{cut-vertex} of $G$ is a vertex $v$ such that $G-v$ has more connected components than $G$.  A \emph{block} of $G$ is a maximal subgraph $H$ such that $H$ has no cut-vertices.  A block is \emph{non-trivial} if it has at least three vertices.  An \emph{ear-decomposition} of $G$ is collection of subgraphs $\{H_0, H_1, \dots, H_k\}$ of $G$ satisfying $G=H_0 \cup H_1 \cup \dots \cup H_k$, $H_0$ is a cycle, and $H_i$ is a path such that |$V(H_i)| \geq 2$ and $V(H_i) \cap \bigcup_{j=0}^{i-1} V(H_j)$ is the set of ends of $H_i$ for all $i \in [k]$. It is well-known that a graph is $2$-connected if and only if it has an ear-decomposition.  The paths $H_1, \dots, H_k$ are the \emph{ears} of the ear-decomposition.  A \emph{theta} is a graph which has an ear-decomposition with exactly one ear.

Given an edge-coloured graph $(G,c)$, a \emph{transversal} of $(G,c)$ is a subgraph $H$ of $G$ such that $V(H)=V(G)$ and $E(H)$ contains exactly one edge of each colour. In particular, a transversal is a rainbow subgraph (which may contain isolated vertices).  

\begin{proof}[Proof of Theorem~\ref{thm:main}]
Suppose the theorem is false and let $(G,c)$ be a counterexample with $|E(G)|$ minimum. By minimality, each colour class contains exactly two edges.   
 We claim that $G$ contains a vertex $v$ such that all edges incident to $v$ have different colours (note that an isolated vertex satisfies this vacuously).  If not, then at each vertex, there is at least one colour that appears twice.  Since there are only $n$ colours, at each vertex there is exactly one colour that appears twice.  For each vertex $v$, let $e_v$ and $f_v$ be the two edges incident to $v$ that have the same colour.  For each $v$, we orient $e_v$ and $f_v$ away from $v$ and apply Theorem~\ref{directed} to find a directed cycle of length at most $\lceil \frac{n}{2} \rceil$.  This corresponds to a rainbow cycle in $G$, which contradicts that $(G,c)$ is a counterexample.

Let $H$ be an arbitrary transversal of $(G,c)$. Since $H$ has $n$ edges and $n$ vertices, it follows that $H$ contains at least one cycle and hence at least one non-trivial block. If $H$ contains two non-trivial blocks, then $H$ contains two rainbow cycles that meet in at most one vertex, and thus a cycle of length at most $\lceil \frac{n}{2} \rceil$. However, this would contradict that $(G,c)$ is a counterexample.  Therefore, $H$ contains exactly one non-trivial block $B$. Suppose $B$ has an ear decomposition with at least two ears.  In this case, $|V(B)| \leq n-2$, since $B$ contains at most $n$ edges. Moreover, $B$ contains a subgraph $B'$ which is either two cycles meeting in at most two vertices, or a subdivision of $K_4$.  If the former holds, then $B'$ contains a cycle of length at most $\lfloor \frac{|V(B')|+2}{2} \rfloor \leq \lceil \frac{n}{2} \rceil$.  If the latter holds, then $B'$ contains four cycles $C_1, \dots, C_4$ such that $\sum_{i \in [4]} |V(C_i)|=2|V(B)'|+4 \leq 2n$.  Thus, one of these four cycles has length at most  $\lceil \frac{n}{2} \rceil$. Since $(G,c)$ is a counterexample, $B$ contains at most one ear.  That is, $B$ is a cycle or a theta.  It follows that every transversal of $(G,c)$ is either a connected graph with exactly one cycle, or the disjoint union of a tree and a graph containing a theta. 

\begin{clm}
$(G,c)$ contains a rainbow theta.
\end{clm}

\begin{proof}[Subproof]
Since $G$ contains a vertex $v$ such that all edges incident to $v$ have different colours, there is a transversal $H$ of $(G,c)$ such that $v$ is an isolated vertex in $H$.  It follows that the other component of $H$ contains a theta. In particular, $(G,c)$ contains a rainbow theta.
\end{proof}

 The rest of the proof only uses the fact that $(G,c)$ contains a rainbow theta.  Let $\theta$ be a rainbow theta in $(G,c)$ with $|V(\theta)|$ minimum.  Let $P_1,P_2$ and $P_3$ be paths in $\theta$ such that $\theta=P_1 \cup P_2 \cup P_3$, $V(P_i) \cap V(P_j):=\{x,y\}$ for all $i \neq j$. and $|V(P_1)| \leq |V(P_2)| \leq |V(P_3)|$.  
 
 \begin{clm} \label{clm:thetabound}
 $\theta$ contains a cycle of length at most $\lfloor \frac{2|V(\theta)|+2}{3} \rfloor$.  Moreover, if $|V(\theta)|=3k+2$, then $\theta$ contains a cycle of length at most $2k+1$, unless $|V(P_1)| = |V(P_2)| = |V(P_3)|$.
 \end{clm}
 
 \begin{proof}[Subproof]
 $P_1 \cup P_2$ is a cycle of length at most $\lfloor \frac{2|V(\theta)|+2}{3} \rfloor$.  Moreover, if $|V(\theta)|=3k+2$, then $P_1 \cup P_2$ is a cycle of length at most $2k+1$, unless $|V(P_1)| = |V(P_2)| = |V(P_3)|$.
 \end{proof}

 A \emph{chord} of $\theta$ is an edge $e \in E(G) \setminus E(\theta)$ such that both ends of $e$ are in $V(\theta)$.    
 \begin{clm} \label{clm:chord}
  $\theta$ has at most two chords. 
 \end{clm}

 \begin{proof}[Subproof]
 Note that $|V(\theta)| \leq n-1$, since $\theta$ is rainbow and thus contains at most $n$ edges.  
Let $e$ be a chord. First suppose $e$ has both endpoints on some $P_i$.  Let $C_i$ be the unique cycle in $P_i \cup \{e\}$.  Note that $C_i$ is rainbow, otherwise $(\theta \setminus E(C_i)) \cup \{e\}$ contradicts the minimality of $|V(\theta)|$. Therefore, $\theta \cup \{e\}$ contains two rainbow cycles meeting in at most two vertices. One of these two cycles has length at most $\lfloor \frac{|V(\theta)|+2}{2} \rfloor \leq \lfloor \frac{n+1}{2} \rfloor  = \lceil \frac{n}{2} \rceil $.  
 
 By symmetry, we may assume that the ends of $e$ are on $P_1$ and $P_2$. Suppose $e$ is coloured red.  If $P_1 \cup P_2$ does not contain a red edge, then $\theta \cup \{e\}$ contains rainbow cycles $C_1, \dots, C_4$ such that $\sum_{i \in [4]} |V(C_i)|=2|V(\theta)|+4 \leq 2n+2$.  Thus, one of these cycles has length at most $\lceil \frac{n}{2} \rceil $.
  It follows that some edge $e'$ of $P_1 \cup P_2$ is also red.  By the minimality of $|V(\theta)|$, this is only possible if $e$ and $e'$ are incident and one end of $e'$ is in $\{x,y\}$.  
  
  If $\theta$ has at least three chords, then by symmetry and the pigeonhole principle, we may assume there exist chords $e_1$ and $e_2$ such that $e_1'$ and $e_2'$ are both incident to $x$. But now $(\theta \cup \{e_1, e_2\}) \setminus \{e_1', e_2'\}$ contains a rainbow theta with fewer vertices than $\theta$.
 \end{proof}

Since $G \setminus E(\theta)$ contains a transversal, there is a rainbow cycle $C$ that is edge-disjoint from $\theta$. Let $V_1=V(\theta) \setminus V(C)$, $V_2=V(\theta) \cap V(C)$, and $V_3=V(C) \setminus V(\theta)$. Since $(G,c)$ is a counterexample, $|V(C)|=|V_2|+|V_3| \geq \lceil \frac{n}{2} \rceil +1 $.  Let $t$ be the number of chords of $\theta$. Note that $t \leq 2$ by Claim~\ref{clm:chord}.  Observe that if $a,b \in V_2$ and $ab \in E(C)$, then $ab$ is a chord of $\theta$.  Therefore, $C[V_2]$ contains at most $t$ edges.  It follows that $|V_2| \leq \lfloor \frac{|V(C)|}{2} \rfloor +t$, or equivalently $|V_3|+t \geq |V_2|$.  Therefore, $|V_3| \geq \frac{ \lceil \frac{n}{2} \rceil +(1-t)}{2}$, and 
\[
|V(\theta)| \leq n-|V_3| \leq n-\frac{ \lceil \frac{n}{2} \rceil +(1-t)}{2} \leq n-\frac{ \lceil \frac{n}{2} \rceil -1}{2},
\]
where the last inequality follows since $t \leq 2$. 

Combining the bound $|V(\theta)| \leq n-\frac{ \lceil \frac{n}{2} \rceil -1}{2}$ with Claim~\ref{clm:thetabound}, we are done unless $n \equiv 2 \pmod 4$ and all the above bounds are tight. In particular, $t=2$, $n=4k+2$, $|V_1|=2k$, $|V_2|=k+2$, $|V_3|=k$.  Moreover, by the second part of Claim~\ref{clm:thetabound}, each of $P_1, P_2$, and $P_3$ contains exactly $k+2$ vertices.  Let $e$ be a chord of $\theta$ and $e'$ be the edge of $\theta$ of the same colour as $e$.  By the second part of Claim~\ref{clm:thetabound}, $\theta':=(\theta \setminus \{e'\}) \cup \{e\}$ contains a cycle of length at most $2k+1=\frac{n}{2}$ vertices, as required. 
\end{proof}

\section{Tightness of the Bound} \label{sec:tight}

We now show that our bound is tight, and that there is a dramatic change of behaviour as we decrease the number of colours from $n$.  To be precise, define the \emph{rainbow girth} of an edge-coloured graph $(G,c)$, denoted $\rg (G,c)$, to be the length of a shortest rainbow cycle in $(G,c)$. If $(G,c)$ does not contain a rainbow cycle, then $\rg (G,c)=\infty$. Let 
\[
f(n,t):=\max \{\rg (G,c) : \text{$|V(G)|=n, |E(G)|=2t$, each colour class of $c$ has size $2$}\}.
\]

\begin{thm} \label{thm:sharp}
For all $n \geq 3$ and $t \leq n$, 
\[
\begin{cases}
f(n, t)=\infty & \text{if $t \leq n-2$,} \\
f(n,t)=n-1 &\text{if $t=n-1$,} \\
f(n,t)= \lceil \frac{n}{2} \rceil &\text{if $t=n$}.
\end{cases}
\]
\end{thm}
\begin{proof}
By Theorem~\ref{thm:main}, $f(n,n) \leq \lceil \frac{n}{2} \rceil$.  For the corresponding lowerbound, let $G$ be a graph with vertex set $\mathbb{Z} / n \mathbb{Z}$ and edges  $i(i+1)$ and $i(i+2)$ for all $i \in V(G)$.  Colour both $i(i+1)$ and $i(i+2)$ with colour $i$ for all $i \in V(G)$. See Figure~\ref{fig:tillgraph}.  It is easy to check that the shortest rainbow cycle in this graph has length $\lceil \frac{n}{2} \rceil$.

\begin{figure}[ht]
\centering
\includegraphics{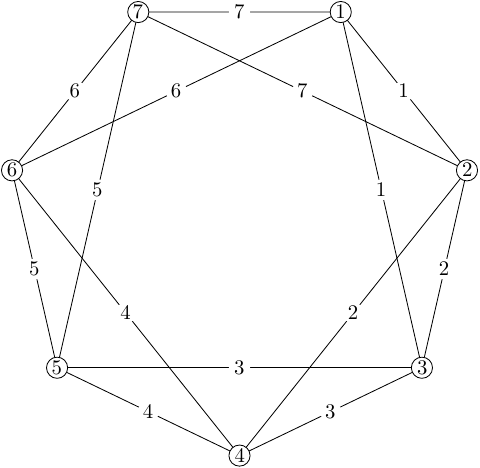}
\caption{The shortest rainbow cycle has length $\lceil \frac{7}{2} \rceil=4$.}
\label{fig:tillgraph}
\end{figure}

We now show $f(n, n-1)=n-1$. For the upperbound, let $G$ be a graph with $|V(G)|=n,|E(G)|=2n-2$, and let $c$ be a colouring of $E(G)$ such that each colour class has size $2$. Since there are only $n-1$
colours, there is a vertex $v$ of $G$ such that all edges incident to $v$ are coloured differently.  Therefore, there is a transversal $H$ of $(G,c)$ such that $v$ is an isolated vertex in $H$.  Since $H-v$ contains $n-1$ vertices and $n-1$ edges, $H-v$ contains a cycle of length at most $n-1$.  For the corresponding lowerbound, let $W_n$ be the wheel graph on $n$ vertices.  Let $c$ be a colouring of $E(W_n)$ such that each colour class is a path with two edges, one of which is incident to the hub vertex. See Figure~\ref{fig:badcolouring}.  Observe that no rainbow cycle of $(W_n, c)$ can use the hub vertex.  Therefore, the shortest rainbow cycle in $(W_n, c)$ has length $n-1$.

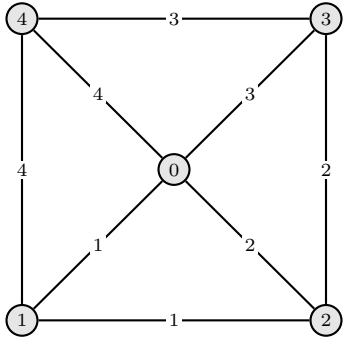
\begin{figure}
\centering
\begin{tikzpicture}[scale=2.0,inner sep=2.0pt]
\tikzstyle{vtx}=[circle,draw,thick,fill=black!10]
\node[vtx] (1) at (0,0) {\tiny $1$};
\node[vtx] (2) at (2,0) {\tiny $2$};
\node[vtx] (3) at (2,2) {\tiny $3$};
\node[vtx] (4) at (0,2) {\tiny $4$};
\node[vtx] (0) at (1,1) {\tiny $0$};

\draw[thick] (1) -- node[fill=white,inner sep=1pt,midway]{\tiny $1$} (2);
\draw[thick] (2) -- node[fill=white,inner sep=1pt,midway]{\tiny $2$} (3);
\draw[thick] (3) -- node[fill=white,inner sep=1pt,midway]{\tiny $3$} (4);
\draw[thick] (4) -- node[fill=white,inner sep=1pt,midway]{\tiny $4$} (1);
\draw[thick] (1) -- node[fill=white,inner sep=1pt,midway]{\tiny $1$} (0);
\draw[thick] (2) -- node[fill=white,inner sep=1pt,midway]{\tiny $2$} (0);
\draw[thick] (3) -- node[fill=white,inner sep=1pt,midway]{\tiny $3$} (0);
\draw[thick] (4) -- node[fill=white,inner sep=1pt,midway]{\tiny $4$} (0);
\end{tikzpicture}
\caption{A colouring $c$ of $E(W_5)$, whose shortest rainbow cycle has length $4$.}
\label{fig:badcolouring}
\end{figure}

By deleting two edges of the same colour from $(W_n, c)$ we obtain a graph on $n$ vertices and $2(n-2)$ edges that does not contain a rainbow cycle.  Therefore, $f(n, t)=\infty$, for all $t \leq n-2$. 
\end{proof}

We have determined $f(n,t)$ exactly for all $t \leq n$.  What happens for $t > n$?
The best general upper bound we can prove follows from a theorem of Bollobás and Szemerédi~\cite{BS02}. To state their result, we need some definitions. The \emph{girth} of a graph $G$, denoted $g(G)$, is the length of a shortest cycle in $G$.  Define
\[
g(n,k):=\max \{g(G): |V(G)|=n, |E(G)|-|V(G)|=k\}.
\]

Bollobás and Szemerédi prove the following.

\begin{thm}[\cite{BS02}] \label{thm:BS}
For all $n \geq 4$ and $k \geq 2$,
\[
g(n,k) \leq \frac{2(n+k)}{3k} (\log k + \log \log k +4).
\]
\end{thm}

As a corollary, we obtain the following.

\begin{thm} \label{thm:rainbowgirth}
For all $n \geq 4$ and $k \geq 2$,
\[
f(n,n+k) \leq \frac{2(n+k)}{3k} ( \log k + \log \log k +4).
\]
\end{thm}

\begin{proof}
Let $G$ be a simple $n$-vertex graph, with $|E(G)|=2(n+k)$ and $c$ be a colouring of $E(G)$ where each colour class has size $2$.  Let $H$ be a transversal of $(G,c)$.  Note that $H$ has $n$ vertices and $n+k$ edges.  By Theorem~\ref{thm:BS}, $H$ contains a cycle of length at most $\frac{2(n+k)}{3k} (\log k + \log \log k +4)$.  Since this cycle is necessarily rainbow, we are done.
\end{proof}

\section{Matroid Generalizations} \label{sec:matroids}

In this section, we consider matroid generalizations of Aharoni's conjecture.  For the reader unfamiliar with matroids, we introduce all the necessary definitions now.  Note that nothing beyond basic linear algebra will be required.  For a more thorough introduction to matroids, we refer the reader to Oxley~\cite{oxley11}. 

A \emph{matroid} is a pair $M=(E, \mathcal C)$ where $E$ is a finite set, called the \emph{ground set} of $M$, and $\mathcal C$ is a collection of subsets of $E$, called \emph{circuits}, satisfying
\begin{itemize}
    \item $\emptyset \notin \mathcal C$, 
    \item if $C'$ is a proper subset of $C \in \mathcal C$, then $C' \notin \mathcal C$,
    \item if $C_1$ and $C_2$ are distinct members of $\mathcal C$ and $e \in C_1 \cap C_2$, then there exists $C_3 \subseteq (C_1 \cup C_2) \setminus \{e\}$.
\end{itemize}
A set $I \subseteq E$ is \emph{independent} if it does not contain a circuit.  The \emph{rank} of $X \subseteq E$ is the size of a largest independent set contained in $X$, and is denoted $r_M(X)$.  The \emph{rank} of $M$ is $r(M):=r_M(E)$.  A matroid is \emph{simple} it it does not contain any circuits of size $1$ or $2$.  We now give examples of all the matroids that appear in this paper.

Let $G$ be a graph.  We will consider two different matroids with ground set $E(G)$.  The circuits of the first matroid are the (edges of) cycles of $G$. This is the \emph{cycle matroid} of $G$, denoted $M(G)$.   A matroid is \emph{graphic} if it is isomorphic to the cycle matroid of some graph.   The second matroid is the dual of the cycle matroid of $G$.  However, we will not define duality, opting instead to define this matroid directly.  An \emph{edge-cut} of $G$ is a set of edges $C^*$ such that $G \setminus C^*$ has more connected components than $G$.  A \emph{cocycle} is an inclusion-wise minimal edge-cut.  The collection of cocycles of $G$ is also a matroid, called the \emph{cocycle matroid} of $G$, and is denoted $M(G)^*$.  A matroid is \emph{cographic} if it is isomorphic to the cocycle matroid of some graph. 

Let $\mathbb F$ be a field.  An \emph{$\mathbb F$-matrix} is a matrix with entries in $\mathbb F$.  Let $A$ be an $\mathbb F$-matrix whose columns are labelled by a finite set $E$. The \emph{column matroid} of $A$, denoted $M[A]$, is the matroid with ground set $E$ whose circuits correspond to the minimal (under inclusion) linearly dependent columns of $A$.  A matroid is \emph{representable over $\mathbb{F}$} if it is isomorphic to $M[A]$ for some $\mathbb{F}$-matrix $A$.  A matroid is \emph{binary} if it representable over the two-element field,  and it is \emph{regular} if it is representable over every field.  

Finally, for integers $0 \leq k \leq n$, the \emph{uniform matroid} $U_{k,n}$ is the matroid with ground set $[n]$, whose circuits are all the subsets of $[n]$ of size $k+1$.  

An attractive feature of Aharoni's conjecture as opposed to the Caccetta-Häggkvist conjecture, is that there is a natural matroid generalization.  For example, the following is the matroid analogue of Theorem~\ref{thm:main}.

\begin{conj} \label{conj:matroid}
Let $M$ be a simple rank-$(n-1)$ matroid and $c$ be a colouring of $E(M)$ with $n$ colours, where each colour class has size at least $2$. Then $M$ contains a rainbow circuit of size at most $\lceil \frac{n}{2} \rceil$.
\end{conj}

 Let $G$ be a simple, connected, $n$-vertex graph, and $r$ be the rank of $M(G)$.  Note that $r$ is the number of edges in a spanning tree of $G$, and so $n-1=r$.  Moreover, since the circuits of $M(G)$ are the cycles of $G$, Conjecture~\ref{conj:matroid} holds for graphic matroids by Theorem~\ref{thm:main}.  

 Unfortunately, it is easy to see that Conjecture~\ref{conj:matroid} is false, since the uniform matroid $U_{n-1, m}$ does not contain \emph{any} circuits of size less than $n$.  On the other hand, we now prove that Conjecture~\ref{conj:matroid} is true for cographic matroids. 

\begin{thm} \label{thm:cographic}
Let $N$ be a simple rank-$(n-1)$ cographic matroid and $c$ be a colouring of $E(N)$ with $n$ colours, where each colour class has size at least $2$. Then $N$ contains a rainbow circuit of size at most $\lceil \frac{n}{2} \rceil$.
\end{thm}

\begin{proof}
Let $(N,c)$ be a counterexample with $|E(N)|$ minimum.  By minimality, each colour class has size exactly $2$.  
  Let $G$ be a graph such that $N=M(G)^*$. Let $G_1, \dots, G_k$ be the connected components of $G$, and $r_i:=r(M(G_i)^*)$ for each $i \in [k]$. Since every cocycle of $N$ is a cocycle in some $N_i:=M(G_i)^*$, it follows that
  \begin{equation} \label{additiverank}
  \sum_{i \in [k]} r_i=r(N)=n-1.  
  \end{equation}
  
  First suppose there is some $j \in [k]$ such that $|E(G_j)| \geq 2(r_j+1)$.  By merging colour classes we may assume that exactly $r_j+1$ colours appear in $E(G_j)$ and each of these colours appears at least twice in $E(G_j)$.  By minimality, $G_j$ contains a rainbow cocycle $C^*$ of size at most $\lceil \frac{r_j+1}{2} \rceil \leq \lceil \frac{n}{2} \rceil$.  By unmerging colours, $C^*$ is also a rainbow cocycle of $G$, so we are done.  By (\ref{additiverank}), such an index $j$ exists unless $k=2, |E(G_1)|=2r_1+1$, and $|E(G_2)|=2r_1+1$.  By (\ref{additiverank}) and symmetry, we may assume $r_1 \leq \lfloor \frac{n-1}{2} \rfloor$.  Since |$E(G_1)|=2r_1+1$ and each colour appears at most twice in $E(G_1)$, there exists a rainbow set $A \subseteq E(G_1)$ such that $|A| = r_1+1$.  Since $|A| > r_1$, $A$ contains a cocycle $C^*$ of $G$.  Since $C^*$ is rainbow and $|C^*| \leq |A|= r_1+1 \leq \lfloor \frac{n-1}{2} \rfloor+1 = \lceil \frac{n}{2} \rceil$, we are done.  
  
Henceforth, we may assume that $G$ is connected.  Let $N^*=M(G)$, the cycle matroid of $G$.  We use the well-known fact that $r(N)+r(N^*)=|E(G)|=2n$.  Therefore, since $N$ has rank $n-1$, $N^*$ has rank $n+1$.  Since $G$ is connected, $|V(G)|=n+2$. 

 For each vertex $v \in V(G)$, let $\delta_G(v)$ be the set of edges of $G$ incident to $v$. Since $\delta_G(v)$ is an edge-cut for each $v \in V(G)$ and $N$ is simple, $G$ has minimum degree at least $3$.   Moreover, since there are exactly $n$ colours and $n+2$ vertices, there are at least two distinct vertices $x$ and $y$ of $G$ such that $\delta_G(x)$ and $\delta_G(y)$ are both rainbow.  If $\deg_G(x) \leq \lceil \frac{n}{2} \rceil$ or $\deg_G(y) \leq \lceil \frac{n}{2} \rceil$, then $\delta_G(x)$ or $\delta_G(y)$ contains a rainbow cocycle of size at most $\lceil \frac{n}{2} \rceil$.  Thus, $\deg_G(x), \deg_G(y) \geq \lceil \frac{n}{2} \rceil+1$. Since $4n=2|E(G)|=\sum_{v \in V(G)} \deg_G(v)$, it follows that 
 $\sum_{v \in V(G) \setminus \{x,y\}} \deg_G(v) \leq 3n-2$.  Therefore, some vertex $z \in V(G) \setminus \{x,y\}$ has degree at most $2$, which contradicts that $G$ has minimum degree at least $3$.
 \end{proof}

We finish this section by giving an infinite family of binary matroids for which Conjecture~\ref{conj:matroid} fails.

\begin{thm} \label{thm:binary}
For each even integer $n \geq 6$, there exists a simple rank-$(n-1)$ binary matroid $M$ on $2n$ elements, and a colouring of $E(M)$ where each colour class has size $2$, such that all rainbow circuits of $(M,c)$ have size strictly greater than $\frac{n}{2}$.
\end{thm}

\begin{proof} 
Let $n \geq 6$ be even.  For each $i \in [n-1]$, let $\mathbf{e}_i$ be the $i$th standard basis vector in $\mathbb{F}_2^{n-1}$. Let $\mathbf 0$ and $\mathbf 1$ be the all-zeros and all-ones vectors in $\mathbb{F}_2^{n-1}$, respectively. Let $M$ be the binary matroid represented by the following $2n$ vectors $\mathcal V$. 

\begin{itemize}
    \item $\mathbf{e}_i$, for all $i \in [n-1]$;
    \item $\mathbf{e}_i+ \mathbf{e}_{i+1}$, for all $i \in [n-2]$;
    \item $\mathbf 1$, $\mathbf{1} + \mathbf{e}_{n-2}$, and $\mathbf{e}_1+ \mathbf{e}_{n-2}$.  
\end{itemize}

Since $\{ \mathbf{e}_i \mid i \in [n-1]\} $ are linearly independent, $M$ has rank $n-1$.  Moreover, all vectors in $\mathcal V$ are distinct and non-zero, so $M$ is simple.  

We now specify the colouring, which is just a pairing of $\mathcal V$.  For each $i \in [n-3]$, we pair $\mathbf{e}_i$ with $\mathbf{e}_i + \mathbf{e}_{i+1}$.  Finally, we pair $\mathbf{e}_{n-2}$ with $\mathbf{e}_1+ \mathbf{e}_{n-2}$; $\mathbf{e}_{n-1}$ with $\mathbf{e}_{n-2}+\mathbf{e}_{n-1}$; and $\mathbf 1$ with $\mathbf{1} + \mathbf{e}_{n-2}$.  To illustrate, the case $n=6$ is given by the following matrix, where column $i$ and column $6+i$ are the same colour for all $i \in [6]$.  

\[
\left( \begin{array}{@{}*{12}{c}@{}}
\textcolor{red}{1} & \textcolor{blue}{0} & \textcolor{green}{0} & \textcolor{magenta}{0} & \textcolor{purple}{0} & \textcolor{cyan}{1}     & \textcolor{red}{1} & \textcolor{blue}{0} & \textcolor{green}{0} & \textcolor{magenta}{1} & \textcolor{purple}{0} & \textcolor{cyan}{1}      \\
\textcolor{red}{0} & \textcolor{blue}{1} & \textcolor{green}{0} & \textcolor{magenta}{0} & \textcolor{purple}{0} & \textcolor{cyan}{1}      & \textcolor{red}{1} & \textcolor{blue}{1} & \textcolor{green}{0} & \textcolor{magenta}{0} & \textcolor{purple}{0} & \textcolor{cyan}{1}      \\
\textcolor{red}{0} & \textcolor{blue}{0} & \textcolor{green}{1} & \textcolor{magenta}{0} & \textcolor{purple}{0} & \textcolor{cyan}{1}      & \textcolor{red}{0} & \textcolor{blue}{1} & \textcolor{green}{1} & \textcolor{magenta}{0} & \textcolor{purple}{0} & \textcolor{cyan}{1}      \\
\textcolor{red}{0} & \textcolor{blue}{0} & \textcolor{green}{0} & \textcolor{magenta}{1} & \textcolor{purple}{0} & \textcolor{cyan}{1}      & \textcolor{red}{0} & \textcolor{blue}{0} & \textcolor{green}{1} & \textcolor{magenta}{1} & \textcolor{purple}{1} & \textcolor{cyan}{0}      \\
\textcolor{red}{0} & \textcolor{blue}{0} & \textcolor{green}{0} & \textcolor{magenta}{0} & \textcolor{purple}{1} & \textcolor{cyan}{1}      & \textcolor{red}{0} & \textcolor{blue}{0} & \textcolor{green}{0} & \textcolor{magenta}{0} & \textcolor{purple}{1} & \textcolor{cyan}{1}      
\end{array} \right)\]

Note that a subset of $\mathcal V$ is linearly dependent if and only if it sums to $\mathbf 0$.  Therefore, it suffices to prove that every rainbow subset of $\mathcal V$ summing to  $\mathbf 0$ has size more than $\frac{n}{2}$.  Let $\mathcal C \subseteq \mathcal V$ be a rainbow set such that $\sum_{v \in \mathcal C} v = \mathbf 0$.  

We first consider the case $\mathbf 1 \in \mathcal{C}$ or $\mathbf{1} + \mathbf{e}_{n-2} \in \mathcal{C}$. Since $\mathbf 1$ and  $\mathbf{1} + \mathbf{e}_{n-2}$ are the same colour, exactly one of them, which we call $x$, is in $\mathcal C$.  Since $\mathbf{e}_{n-1}$ and $\mathbf{e}_{n-2}+\mathbf{e}_{n-1}$ are the only other vectors in $\mathcal{V}$ that are non-zero in their $(n-1)$th coordinate, exactly one of them, which we call $y$, is in $\mathcal{C}$. Let $\mathcal C' = \mathcal C \setminus \{x,y\}$.  In all four cases, $\sum_{v \in \mathcal C '}$ is $\mathbf 1 + \mathbf{e}_{n-1}$ or $\mathbf 1 + \mathbf{e}_{n-1}+\mathbf{e}_{n-2}$. Since $n$ is even, and all vectors in $\mathcal{C}'$ have support at most $2$, $|\mathcal C '| \geq \frac{n}{2}-1$.  Therefore, $|\mathcal{C}| > \frac{n}{2}$.

The remaining case is $\mathbf 1 \notin \mathcal{C}$ and $\mathbf{1} + \mathbf{e}_{n-2} \notin \mathcal C$.  The only other vectors in $\mathcal V$ whose $(n-1)$th coordinate is non-zero are $\mathbf{e}_{n-1}$ and $\mathbf{e}_{n-2}+\mathbf{e}_{n-1}$.  Since  $\mathbf{e}_{n-1}$ and $\mathbf{e}_{n-2}+\mathbf{e}_{n-1}$ are the same colour, they cannot both be in $\mathcal{C}$.  Thus, $\mathbf{e}_{n-1} \notin \mathcal C$ and $\mathbf{e}_{n-2}+\mathbf{e}_{n-1} \notin \mathcal C$. Let $\mathcal C=\mathcal C_1 \cup \mathcal C_2$, where $\mathcal C_i$ are the vectors in $\mathcal{C}$ whose support has size $i$. Let $G$ be the cycle with vertex set $\mathbb{Z}_{n-2}$ and edge set $\{i(i+1): i \in \mathbb{Z}_{n-2}\}$.  Note that we may regard $\mathcal{C}_1 \subseteq V(G)$ and $\mathcal{C}_2 \subseteq E(G)$. Let $H$ be the subgraph of $G$ with vertex set $V(G)$ and edge set $\mathcal{C}_2$.  Since $\sum_{v \in \mathcal C} v = \mathbf 0$, the set of odd-degree vertices of $H$ is precisely $\mathcal C_1$.   Thus, $|\mathcal C_1|$ is even, and there are only two possibilites for $H$. If $|\mathcal C_1| \geq 2$, then there exists $i \in \mathbb{Z}_{n-2}$ such that $i \in \mathcal C_1$ and $i(i+1) \in \mathcal{C}_2$.  However, this contradicts that $\mathcal{C}$ is rainbow.  Therefore, $\mathcal C_1=\emptyset$ and $\mathcal{C}_2=|E(G)|$. Thus, $|\mathcal C| = |\mathcal{C}_2|= n-2 > \frac{n}{2}$.
\end{proof}

 A slight modification of the above construction also yields counterexamples for all odd integers $n \geq 7$.  On the other hand, it is fairly easy to show that Conjecture~\ref{conj:matroid} holds for binary matroids when $n \leq 5$ (it is true vacuously when $n \leq 4$). Thus, Conjecture~\ref{conj:matroid} holds for binary matroids if and only if $n \leq 5$.  
 
 \textbf{Update.} A forthcoming paper of Yuhang Bai, Nathan Bowler, and Tony Huynh strengthens Theorem~\ref{thm:binary} as follows.  For each $n \geq 1$, there exists a rank-$6n$ simple binary matroid $M$ on $12n+2$ elements, and a partition of $E(M)$ into $6n+1$ colour classes of size $2$, such that the shortest rainbow circuit of $M$ has size $4n$.

\section{Open Problems} \label{sec:open}
Note that in proving Theorem~\ref{thm:rainbowgirth}, we only use one fixed transversal.  By considering multiple transversals, we suspect that the bound in Theorem~\ref{thm:rainbowgirth} can be improved.  

\begin{prob}
Determine $f(n,t)$ for $t > n$.
\end{prob}

Since the Caccetta-Häggkvist conjecture is known to hold for $r \in \{3,4,5\}$, another possible direction is to prove Aharoni's conjecture for $r \in \{3,4,5\}$.

\begin{prob}
Prove that Conjecture~\ref{conj:aharoni} (or Conjecture~\ref{conj:weakaharoni}) holds for $r \in \{3,4,5\}$.
\end{prob}

Recall that our proof of Aharoni's conjecture for $r=2$ uses Theorem~\ref{directed} as a blackbox.  It would be interesting to find a proof of Theorem~\ref{thm:main} that avoids using Theorem~\ref{directed}. 

Finally, by Theorems~\ref{thm:main} and~\ref{thm:cographic}, Conjecture~\ref{conj:matroid} holds for both graphic and cographic matroids. Therefore, we suspect there is a proof of Conjecture~\ref{conj:matroid} for regular matroids via Seymour's regular matroid decomposition theorem~\cite{Seymour80}.

\begin{conj} \label{conj:regular}
Let $M$ be a simple rank-$(n-1)$ regular matroid and $c$ be a colouring of $E(M)$ with $n$ colours, where each colour class has size at least $2$. Then $M$ contains a rainbow circuit of size at most $\lceil \frac{n}{2} \rceil$.
\end{conj}

\subsection*{Acknowledgements.} We would like to thank Ron Aharoni for bringing Conjecture~\ref{conj:aharoni} to our attention.  We also thank Tillmann Miltzow for help in making Figure~\ref{fig:tillgraph}.  

After the journal version of this paper was published (see \url{https://onlinelibrary.wiley.com/doi/10.1002/jgt.22607}), Emanuele Natale and Édouard Oyallon found a slight inaccuracy in the proof of Theorem~\ref{thm:binary}.  We thank them both for alerting us and have corrected the proof of Theorem~\ref{thm:binary} in this version.

\bibliographystyle{plain}
\bibliography{references} 
\end{document}